\documentclass{article}
\usepackage[latin1]{inputenc} 
\usepackage[T1]{fontenc}
\usepackage[english]{babel}
\usepackage[dvipsnames]{xcolor} 
\usepackage{framed} 
\colorlet{shadecolor}{red!20}
\usepackage{indentfirst} 
\usepackage{latexsym}
\usepackage{amssymb}
\usepackage{amsmath}
\usepackage{amsthm}
\usepackage{url}
\usepackage{enumitem}
\usepackage{amssymb,latexsym,amsmath,epsfig,amsthm} 

\newtheorem{theorem}{Theorem}[section]
\newtheorem{corollary}[theorem]{Corollary}
\newtheorem{lemma}[theorem]{Lemma}
\newtheorem{proposition}[theorem]{Proposition}
\newtheorem{definition}[theorem]{Definition}
\newtheorem*{theorem*}{Theorem}
\theoremstyle{definition}
\newtheorem{example}[theorem]{Example}

\newcommand{\bzero}{\beta(0,1)_{d}}
\newcommand{\bR}{\beta\mathbb{R}_{d}}
\newcommand{\Pol}{P\left(x_{1},\dots,x_{n}\right)}
\newcommand{\U}{\mathcal{U}}
\newcommand{\N}{\mathbb{N}}
\newcommand{\Z}{\mathbb{Z}}
\newcommand{\Q}{\mathbb{Q}}
\newcommand{\R}{\mathbb{R}}
\newcommand{\V}{\mathcal{V}}

\begin{document}

\begin{center}
\uppercase{\bf Partition regularity of polynomial systems near zero}
\vskip 20pt
{\bf Lorenzo Luperi Baglini\footnote{Dipartimento di Matematica, Universit\`{a} di Milano, Via Saldini 50, 20133 Milano, Italy, supported by grant P30821-N35 of the Austrian Science Fund FWF.}}\\

{\tt lorenzo.luperi@unimi.it}\\
\end{center}

\centerline{\bf Abstract}
\noindent 

Recently, S.~Kanti Patra and Md.~Moid Shaik proved the existence of monochromatic solutions to systems of polynomial equations near zero for particular dense subsemigroups $S$ of $((0,\infty),+)$. We extend their results to a much larger class of systems whilst weakening the requests on $S$, using solely basic results about ultrafilters.

\vspace{0.5cm}


\section{Introduction}

The study of the partition regularity of linear systems of equations started in the early twentieth century, with the works of Schur, Van Der Waerden and Rado. Let us start by recalling the basic definition.

\begin{definition} Let $S\subseteq\R$. Let $P_{1}\left(x_{1},\dots,x_{n}\right)$,$\dots$,$P_{m}\left(x_{1},\dots,x_{n}\right)\in \R\left[x_{1},\dots,x_{n}\right]$. Let 
\[\sigma\left(x_{1},\dots,x_{n}\right)=\begin{cases}  P_{1}\left(x_{1},\dots,x_{n}\right),\\ \hspace{1.2cm}\vdots \\ P_{m}\left(x_{1},\dots,x_{n}\right). \end{cases}\]

We say that the system of equations $\sigma\left(x_{1},\dots,x_{n}\right)=(0,\dots,0)$ is\footnote{From now on, we will simply write $\sigma\left(x_{1},\dots,x_{n}\right)=0$ to simplify the notation.} partition regular on $S$ if it has a monochromatic solution\footnote{In other papers, strengthened versions of this notion where also some additional properties on the solution, like it being non constant or injective, have been considered. Although minor modifications of our proofs would work also for these strengthened notions, we prefer to use the basic definition of partition regularity so not to have to handle unnecessary complications, since our goal is to show a method to prove the partition regularity near zero of systems that are partition regular on $\R$ or $\Q$.} in every finite coloring of $S\setminus\{0\}$, namely if for every natural number $r$, for every partition $S=\bigcup\limits_{i=1}^{r}A_{i}$, there is an index $j\leq r$ and numbers $a_{1},\dots,a_{n}\in A_{j}$ such that $\forall j\in\{1,\dots,m\} \ P_{j}\left(a_{1},\dots,a_{n}\right)=0$.

\end{definition}

The case of linear systems was settled by Richard Rado (see \cite{rif18}; we recall here the version from \cite{new york}, which covers the partition regularity over $\N,\Z,\R^{+},\R$) in terms of the so-called "columns condition", that we recall.

\begin{definition}\label{columns condition}Let $u,v\in\N$, let $S\in\{\N,\Z,\R^{+},\R\}$, let $F=\Q$ if $S=\N$ or $S=\Z$, and let $F=\R$ if $S=\mathbb{R}^{+}$ or $S=\R$. Let $A$ be a $u\times v$ matrix with entries from $F$, with columns $c_{1},\dots,c_{v}$. We say that $A$ satisfies the columns condition over $F$ if there exists a partition $\{I_{1},\dots,I_{m}\}$ of $\{1,\dots,v\}$ such that 
\begin{itemize}
\item $\sum_{i\in I_{1}}\vec{c_{i}}=\vec{0}$;
\item for each $t\in\{2,\dots,m\}$ (if any) $\sum_{i\in I_{t}}\vec{c_{i}}$ is a linear combination over $F$ of $\{\vec{c_{i}}\mid i\in\bigcup_{k=1}^{t-1}I_{k}\}$.
\end{itemize}
\end{definition}

\begin{theorem}[Rado]\label{RadoS} Let $u,v\in\N$, let $S\in\{\N,\Z,\R^{+},\R\}$, let $F=\Q$ if $S=\N$ or $S=\Z$, and let $F=\R$ if $S=\R^{+}$ or $S=\R$. Let $A$ be a $u\times v$ matrix with entries from $F$, with columns $c_{1},\dots,c_{v}$. Then the system $A\vec{x}=0$ is partition regular over $S$ if and only if $A$ satisfies the columns condition over $F$.\end{theorem}

For linear equations on $\R$, Rado's result reads as follows:

\begin{theorem}[Rado]\label{Rado}Let $\Pol= \sum_{i=1}^{n}a_{i}x_{i}\in\mathbb{R}\left[x_{1},\dots,x_{n}\right]$ be a linear polynomial with nonzero coefficients. The following conditions are equivalent:
\begin{enumerate}
	\item The equation $\Pol=0$ is partition regular on $\mathbb{R}$;
	\item there is a nonempty subset $J$ of $\{1,\dots,n\}$ such that $\sum\limits_{j\in J}a_{j}=0$.
\end{enumerate}
\end{theorem}

Let us notice that an immediate, but interesting, consequence of Theorem \ref{Rado} is the following\footnote{This is a particular instance of Theorem 2.1 in \cite{lefmann}; we present a proof here as it is very simple.}:

\begin{theorem}\label{Rado2} Let $u,v\in\N$. Let $A$ be a $u\times v$ matrix with entries from $\R$, with columns $c_{1},\dots,c_{v}$. Let $r\in \R\setminus\{0\}$, and let $\vec{x}^{r}:=\left(x_{1}^{r},\dots,x_{m}^{r}\right)$. Then the system $A\vec{x}^{r}=0$ is partition regular over $\R^{+}$ if and only if $A$ satisfies the columns condition over $\R$.\end{theorem}

\begin{proof} Let $\R^{+}=A_{1}\cup\dots\cup A_{k}$ be a finite coloring of $\R^{+}$. As $f_{r}(x):=x^{\frac{1}{r}}:\R^{+}\rightarrow \R^{+}$ is a bijection, we can consider the coloring $\R^{+}=B_{1}\cup\dots\cup B_{k}$ where for each $i\leq k$ one sets $B_{i}:=\{f_{r}(x)\mid x\in A_{i}\}$. It holds that the system $A\vec{x}^{r}=0$ has a monochromatic solution with respect to coloring $A_{1}\cup\dots\cup A_{k}$ if and only if the linear system $A\vec{x}$ has a monochromatic solution with respect to coloring $B_{1}\cup\dots\cup B_{k}$.\end{proof}

For example, the Pythagorean equation $x^{2}+y^{2}=z^{2}$ is partition regular on $\R$, whilst it is an open problem if it is partition regular on $\N$ or not; also, Fermat equation $x^{3}+y^{3}=z^{3}$ is partition regular on $\R$, whilst it does not even have solutions on $\N$. For results about the partition regularity of similar kinds of diagonal Diophantine equations on $\N$ we refer to \cite{sean}.

The general nonlinear case is much more complicated. For example, whilst, as expected, a linear system is partition regular on $\R$ if and only if it is partition regular on $\R^{+}$, the same does not hold for nonlinear systems, or even nonlinear equations: the equation $x_{1}x_{2}+x_{3}=0$ is partition regular on $\R$ (even if we add the natural request that $x_{1},x_{2},x_{3}$ should be mutually distinct), but it is clearly not partition regular on $\R^{+}$. 

In \cite{rif6}, P.~Csikv\'ari, K.~Gyarmati and A.~S\'arközy asked wheter every finite coloring of $\N$ contains monochromatic $a,b,c,d$ such that $a+b=cd$. The question was answered positively by N.~Hindman in \cite{rif13} (see also \cite{bergelson} for a short proof) as a particular case of the following result\footnote{The main result in \cite{rif13} is Theorem 5, which is more general than the result that we recall here.}:

\begin{theorem}[Hindman]\label{ConsHind} For all natural numbers $n,m\geq 1$ the equation \[\sum\limits_{i=1}^{n}x_{i}-\prod\limits_{j=1}^{m}y_{j}=0\] is partition regular on $\N$. \end{theorem}

Theorem \ref{ConsHind} was the major motivation for our research in \cite{mio}. In this paper, we showed that Theorem \ref{ConsHind} could be seen as a particular case of a more general characterization of partition regular equations\footnote{In \cite{mio} we showed the partition regularity of polynomials belonging to a much larger class $\mathcal{C}$ but, since the definition of $\mathcal{C}$ is rather involved, we limit ourself here to present a simpler generalization of Theorem \ref{ConsHind}.}, that we proved by means of nonstandard methods. The same result is proven in \cite{advances} by means of purely standard methods based on ultrafilters. To introduce the result, we first need a definition: 

\begin{definition}\label{Qi} Let $m$ be a positive natural number and let $\{y_{1},\dots,y_{m}\}$ be a set of mutually distinct variables. For all $F\subseteq\{1,..,m\}$ we denote by $Q_{F}(y_{1},\dots,y_{m})$ the monomial
\begin{center} $Q_{F}(y_{1},\dots,y_{m})=\begin{cases} \prod\limits_{j\in F} y_{j}, & \mbox{if  } F\neq \emptyset;\\ 1, & \mbox{if  } F=\emptyset.\end{cases}$ \end{center}
\end{definition}

\begin{theorem}\label{lev} Let $n\geq 2$ be a natural number, let $R\left(x_{1},\dots,x_{n}\right)=\sum\limits_{i=1}^{n} a_{i}x_{i}\in\Z\left[x_{1},\dots,x_{n}\right]$ be a partition regular polynomial on $\N$, and let $m$ be a positive natural number. Then, for all $F_{1},\dots,F_{n}\subseteq\{1,\dots,m\}$, the polynomial
\begin{center} $P\left(x_{1},\dots,x_{n},y_{1},\dots,y_{m}\right)=\sum\limits_{i=1}^{n} a_{i}x_{i}Q_{F_{i}}\left(y_{1},\dots,y_{m}\right)$ \end{center}
is partition regular. \end{theorem}

Notice that Theorem \ref{ConsHind} is (apart the renaming $y_{1}:=x_{n+1}$) a particular case of Theorem \ref{lev} with $R\left(x_{1},\dots,x_{n},x_{n+1}\right):=\sum_{i=1}^{n}x_{i}-x_{n+1}, F_{1}=\dots=F_{n}=\emptyset, Q_{F_{n+1}}=\prod\limits_{j=2}^{m}y_{j}$.

Several other results about the partition regularity of equations on $\N$ have been proven in the past few years; we refer to \cite{advances} for an overview, and to \cite{sean} for the latest results we are aware of.

A refined version of the partition regularity of polynomial systems on $\mathbb{R}$, namely the partition regularity near zero, appeared in \cite{HL}. We give the definition for a general $S\subseteq (0,1)$, although in the following we will assume some additional algebraic properties on $S$.

\begin{definition}\label{PR near 0} Let $P_{1}\left(x_{1},\dots,x_{n}\right)$,$\dots$,$P_{m}\left(x_{1},\dots,x_{n}\right)\in\mathbb{R}\left[x_{1},\dots,x_{n}\right]$, and let $S\subseteq (0,1)$. Let
\[\sigma\left(x_{1},\dots,x_{n}\right)=\begin{cases}  P_{1}\left(x_{1},\dots,x_{n}\right),\\ \hspace{1.2cm}\vdots \\ P_{m}\left(x_{1},\dots,x_{n}\right). \end{cases}\]
We say that the system of equations $\sigma\left(x_{1},\dots,x_{n}\right)=0$ is partition regular near zero with respect to $S$ if for all $\varepsilon > 0$, for all finite partitions $S=A_{1}\cup\dots\cup A_{k}$ there exists $j\leq k$ and $a_{1},\dots,a_{n}\in A_{j}\cap (0,\varepsilon)$ such that $\sigma\left(a_{1},\dots,a_{n}\right)=0$.
\end{definition}

Notice that, whilst all systems that are partition regular near zero are also partition regular on $\R^{+}$, the converse is not always true. As a trivial example, the equation $x=1$ is partition regular on $\R^{+}$ for obvious reasons, but it is not partition regular near zero. Anyhow, in all the following, we will always be interested in systems where all equations have no constant term\footnote{At least for linear systems, the case where there is a nonzero constant term can be reduced to the zero constant term case, as proven by Rado in \cite{rif18}.}.

In the very recent paper \cite{indiani} the authors, building on previous work by N.~Hindman and I.~Leader in \cite{HL}, proved the partition regularity near zero of certain linear polynomial systems, as well as the partition regularity near zero of the equation $a+b=cd$. Let us recall some definitions\footnote{In \cite{indiani}, the definition of $AP$-set is weaker, as it does not require that $d\in A$ as we do, but this does not make any difference in their main results.}.

\begin{definition} Let $A\subseteq\R$. Then $A$ is said to be 
\begin{itemize} 
\item an $IP_{0}$-set if and only if for each $m\in\N$ there exists a finite sequence $\langle y_{n}\rangle_{n=1}^{m}\subseteq\R$ such that $FS\left(\langle y_{n}\rangle_{n=1}^{m}\right)\subseteq A$;
\item an $AP$-set if and only if for each $k\in\N$ there are $a,d\in \R$ such that $a,d,a+d,\dots,a+kd\in A$.
\end{itemize}\end{definition}

Notice that both properties tell us that $A$ contains a solution to a particular linear homogeneous system.

\begin{definition}[\cite{indiani}, Definition 5]\label{HL semi} Let $(S,+)$ be a dense subsemigroup of $\left( \R^{+},+\right)$. Then $(S,+)$ is an $HL$ semigroup if and only if $(S\cap (0,1),\cdot)$ is a subsemigroup of $((0,1),\cdot)$ and, for each $y\in S\cap (0,1)$ and for each $x\in S$, $\frac{x}{y}\in S$ and $yx\in S$. \end{definition}

The next theorem summarizes the two main results of \cite{indiani}: item (1) is the content of \cite[Theorem 12]{indiani}, and item (2) is the content of \cite[Theorem 14]{indiani}.
\begin{theorem}\label{teo indiani} Let $S\subseteq \R^{+}$. Then
\begin{enumerate}
\item if $(S,+)$ is a dense subsemigroup of $\left( \R^{+},+\right)$ such that $(S\cap (0,1),\cdot)$ is a subsemigroup of $((0,1),\cdot)$, then piecewise syndetic sets in $\left(S\cap (0,1),\cdot\right)$ are $IP_{0}$ and $AP$-rich near zero on $S$;
\item if $S$ is a $HL$ semigroup then equations $\sum_{t=1}^{n}x_{t}=\prod_{t=1}^{n}y_{t}$ are partition regular near zero on $S$.
\end{enumerate}
\end{theorem}

Our goal is to generalize Theorem \ref{teo indiani} in two directions: first, we aim to relax the hyphoteses on $S$; second, we want to prove the partition regularity near zero of a much larger class of polynomial systems. Finally, we want to show that partition regularity near zero and arbitrary partition regularity are closely related. These results will be obtained in Section \ref{applicazioni}, whilst in Section \ref{ult} we recall all the basic results about ultrafilters that we need.

\section{Ultrafilters near zero}\label{ult}

In this paper we assume the reader to be familiar with the fundamental 
properties of the space $\beta S$ of ultrafilters on a discrete\footnote{For this reason, in all this paper we assume $\R$ and its subsets to be endowed with the discrete topology.} semigroup $(S,\cdot)$. We refer to \cite{rif12} for a comprehensive introduction to $\beta S$ and its algebra. Here, we fix some notations and recall only the results that we need. By defining for each $A\subseteq S$,
\begin{equation*}
  \overline{A} = \{\U\in\beta S: A\in U\},
\end{equation*}
$\beta S$ becomes a compact Hausdorff topological space for which $\mathcal{B} = \{\overline{A}:A\subseteq S\}$ is a base of open-and-closed sets. By identifying each element $s\in S$ with the principal ultrafilter $\U_{s}:=\{A\subseteq S: s\in A\}$, $S$ is embeddable into $\beta S$ as a dense subspace. We will also use the following convention: if $A\subseteq B$, we will write $\beta A\subseteq \beta B$, identifying every $\U\in\beta A$ with its extension to $B$, namely with the ultrafilter $\{X\subseteq B\mid \exists Y\in\U \ \text{such that} \ Y\subseteq X\}$. Notice that, in this identification, $\beta A$ is identified with $\{\U\in\beta B\mid A\in\U\}$.

In general, when $(S,\cdot)$ is a semigroup, $\beta S$ can be made into a right-topological semigroup by the operation $\odot$ defined as
\begin{equation*}
  \U\odot \V = \big\{A\subseteq S: \{s\in S: \{t\in S: s\cdot t\in A\}\in \V\}\in\U\big\}.
\end{equation*}

A well known fact that we will often use is that $\left(\beta S,\odot\right)$ has a unique smallest bilateral ideal, that will be denoted by $K\left(\beta S,\odot\right)$. Such ideals always contain an ultrafilter $\U$ which is idempotent, namely $\U$ such that $\U\odot \U = \U$. 

As we are interested in the notion of partition regularity near zero, and for reasons that will be made precise in Proposition \ref{ultra part zero}, we will often talk about ultrafilters in the set $0^{+}(S)$, that is defined as follows\footnote{At the best of our knowledge, $0^{+}(S)$ has first been defined by Hindman and Leader in \cite{HL}, although only for sets $S$ so that $(S,+)$ is a dense subsemigroup of $(0,+\infty)$.}.

\begin{definition} Let $S\subseteq (0,1)$. We let 
\[0^{+}(S)=\{\U\in\beta S\mid \forall \varepsilon >0 \ (0,\varepsilon)\cap S\in\U\}.\] \end{definition}

For simplicity, we let \[0^{+}((0,1))=0^{+}.\]

Notice that, for all $S\subseteq (0,1)$, $0^{+}(S)=0^{+}\cap \beta S$. 

\begin{definition} Let $S\subseteq\R^{+}$. We say that $S$ is infinitesimal if $0\in cl(S)$. \end{definition}

Clearly, $0^{+}(S)\neq\emptyset$ if and only if $S$ is infinitesimal\footnote{The terminology comes from nonstandard analysis: $S$ is infinitesimal if and only if its nonstandard extensions contain infinitesimals.}.

Notice that, in general, any subsemigroup of $((0,1),\cdot)$ is infinitesimal: if $s\in S$, then $s^{n}\in S$ for every $n\in\N$, and $\lim\limits_{n\rightarrow +\infty}s^{n}=0$.

\begin{theorem} Let $(S,\cdot)$ be a subsemigroup of $((0,1),\cdot)$. The following facts hold:
\begin{enumerate}
\item $0^{+}(S)\neq\emptyset$;
\item $0^{+}(S)$ is a closed bilateral ideal of $\left(\beta S,\odot\right)$;
\item $K(\beta S,\odot)\subseteq 0^{+}(S)$.
\end{enumerate}

Moreover, the following facts are equivalent:
\begin{enumerate}
\item[(i)] $S$ is piecewise syndetic in $((0,1),\cdot)$;
\item[(ii)] $\exists\U\in K\left(\bzero,\odot\right)$ such that $S\in\U$;
\item[(iii)] $K\left(\bzero,\odot\right)\cap 0^{+}(S)=K\left(\beta S,\odot\right)$;
\item[(iv)] $K\left(\bzero,\odot\right)\cap 0^{+}(S)\neq\emptyset$.
\end{enumerate}
\end{theorem}

\begin{proof} $(1)$ We already observed that, in this case, $S$ is infinitesimal, hence $0^{+}(S)\neq\emptyset$.

$(2)$ To show that $0^{+}(S)$ is closed it sufficies to notice that 

\[0^{+}(S)=\bigcap_{\varepsilon>0}\overline{(0,\varepsilon)\cap S}.\]

Let us prove that it is a bilateral ideal of $K(\beta S,\odot)$. Let $\U\in 0^{+}(S),\V\in \beta S$. Let $\varepsilon>0$. To prove that $(0,\varepsilon)\cap S\in \V\odot\U$ let us notice that, for all $r\in (0,1)\cap S$, $J_{r}=\{s\in S\mid sr\in (0,\varepsilon)\cap S\}\supseteq \{s\in S\mid s\in\left(0,\varepsilon\right)\}\in\U$ as $r\in (0,1)\cap S$, $\U\in 0^{+}(S)$, hence
\[\{r\in S\mid J_{r}\in\U\}\supseteq (0,1)\cap S\in\V.\]
To prove that $(0,\varepsilon)\cap S\in \U\odot\V$ let us notice that, for all $s\in(0,1)$ with $s\leq\varepsilon$, the set $I_{s}=\{r\in S\mid sr\in (0,\varepsilon)\cap S\}=S$, hence
\[\{s\in S\mid I_{r}\in\V\}\supseteq (0,\varepsilon)\cap S\in\U.\]

$(3)$ This is a straighforward consequence of $(2)$, as $K(\beta S,\odot)$ is by definition the smallest bilateral ideal of $(\beta S,\odot)$ with respect to inclusion.

\vspace{\baselineskip}
Let us now prove the equivalence between facts $[i]-[iv]$.
\vspace{\baselineskip}

$(i)\Leftrightarrow (ii)$ This is a particular case of \cite[Theorem 4.40]{rif12}.

$(ii)\Rightarrow (iii)$ As we identify $\overline{S}$ with $\beta S$, we have that $\beta S\cap K\left(\bzero,\odot\right)\neq\emptyset$, so by \cite[Theorem 1.65.(1)]{rif12} $S\cap K\left(\bzero,\odot\right)=K\left(\beta S,\odot\right)$. By $(ii)$, $K\left(\beta S,\odot\right)\subseteq 0^{+}(S)$, so 
\begin{align*} K\left(\beta S,\odot\right)& =0^{+}(S)\cap K\left(\beta S,\odot\right) \\ & =0^{+}(S)\cap \beta S \cap K\left (\beta (0,1),\odot\right) \\ &= 0^{+}(S)\cap K(\beta (0,1),\odot).\end{align*}

$(iii)\Rightarrow (iv)$ This is trivial.

$(iv)\Rightarrow (ii)$ Trivially, if $\U\in K\left(\bzero,\odot\right)\cap 0^{+}(S)$ then $S\in\U$. \end{proof}

In Section \ref{applicazioni} we will work with two kinds of semigroups $(S,\cdot)$, although with very similar methods based on ultrafilters.

The first kind are subsemigroups $(S,\cdot)$ of $((0,1),\cdot)$ with $S$ piecewise syndetic in $((0,1),\cdot)$. The second are the $\mathbb{Q}$-infinitesimal semigroups, which are defined as follows.

\begin{definition} Let $S\subseteq (0,1)$. We say that is subsemigroup $(S,\cdot)$ of $((0,1),\cdot)$ is $\mathbb{Q}$-infinitesimal if for all $s\in S$, for all $q\in\mathbb{Q}^{+}$, if $qs<1$ then $qs\in S$. \end{definition}

In \cite{indiani}, Kanti Patra and Moid Shaik framed their results about the partition regularity near zero in terms of $HL$-semigroups (see Definition \ref{HL semi}). The setting of $\mathbb{Q}$-infinitesimal semigroups is more general, in the following sense.

\begin{proposition} If $(S,+)$ is an $HL$-semigroup, then $(S\cap (0,1),\cdot)$ is $\mathbb{Q}$-infinitesimal.\end{proposition}

\begin{proof} As $S$ is an $HL$-semigroup, $(S\cap (0,1),\cdot)$ is a semigroup. Let $s\in S\cap (0,1)$ and let $q=\frac{n}{m}\in\mathbb{Q}^{+}$, with $qs<1,n,m\in\N$. Let $k\in\N$ be such that $s^{k}m<1$. Now, $s^{k},s^{k+1}\in S\cap (0,1)$ as $(S\cap (0,1),\cdot)$ is a semigroup, $s^{k}m \in S\cap (0,1)$ as $(S,+)$ is a semigroup and $s^{k}m<1$ and $s^{k}n\in S$ as $s^{k}\in S$ and $(S,+)$ is a semigroup. Hence $\frac{s^{k+1}n}{s^{k}m}=s\frac{n}{m}\in S\cap (0,1)$ as $S$ is an $HL$-semigroup.\end{proof}

On the contrary, not all $\Q$-infinitesimal semigroups $(S,\cdot)$ are of the form $(T\cap (0,1),\cdot)$ for some $HL$-semigroup $T$. For example, take $$S=\{q\pi^{z}\mid q\in\mathbb{Q}^{+},z\in\omega\}\cap (0,1).$$ Then $S$ is clearly $\mathbb{Q}$-infinitesimal, but it is not of the form $(T\cap (0,1),\cdot)$ for some $HL$-semigroup $T$. In fact, otherwise, as $\frac{1}{2},\frac{\pi}{8}\in T\cap (0,1)$ and $\frac{1}{2}+\frac{\pi}{8}\in (0,1)$, we would find $q\in\mathbb{Q}^{+}, z\in\omega$ such that $$\frac{1}{2}+\frac{\pi}{8}=q\pi^{z},$$ against the trascendence of $\pi$.

Finally, let us observe that being piecewise syndetic in $\left((0,1),\cdot\right)$ or being $\mathbb{Q}$-infinitesimal are distinct notions: $\mathbb{Q}^{+}$ is $\mathbb{Q}$-infinitesimal but $\mathbb{Q}\cap (0,1)$ is not piecewise syndetic\footnote{We will provide an alternative proof of this fact in Example \ref{razionali}.} in $(0,1)$; conversely, $\left(0,\frac{1}{2}\right)$ is piecewise syndetic in $((0,1),\cdot)$ but it is not $\mathbb{Q}$-infinitesimal.

Is is well known that ultrafilters and their algebra provide an useful tool to study partition regular properties. We specify this well known general result (see for example \cite[Theorem 5.7]{rif12}) to our present framework:

\begin{theorem}\label{aaaa} Let $(S,\cdot)$ be a subsemigroup of $\left(\R^{+},\cdot\right)$. Let $P_{1}\left(x_{1},\dots,x_{n}\right),\dots,P_{m}\left(x_{1},\dots,x_{n}\right)\in \R\left[x_{1},\dots,x_{n}\right]$. Let 
\[\sigma\left(x_{1},\dots,x_{n}\right)=\begin{cases}  P_{1}\left(x_{1},\dots,x_{n}\right),\\ \hspace{1.2cm}\vdots \\ P_{m}\left(x_{1},\dots,x_{n}\right). \end{cases}\]
Then the system $\sigma\left(x_{1},\dots,x_{n}\right)=0$ is partition regular on $S$ if an only if there exists $\U\in\beta S$ such that for all $A\in\U$ there are $a_{1},\dots,a_{n}\in A$ with $\sigma\left(a_{1},\dots,a_{n}\right)=0$. \end{theorem}

\begin{definition} Under the conditions of Theorem \ref{aaaa}, we say that $\U$ witnesses the partition regularity of the system $\sigma\left(x_{1},\dots,x_{n}\right)=0$, and we call it a $\iota_{\sigma}$-ultrafilter. \end{definition}

Let us observe that, when specified to the partition regularity near zero, Theorem \ref{aaaa} reads as follows.

\begin{proposition}\label{ultra part zero} Let $(S,\cdot)$ be a subsemigroup of $\left((0,1),\cdot\right)$. Let $P_{1}\left(x_{1},\dots,x_{n}\right)$,$\dots$,$P_{m}\left(x_{1},\dots,x_{n}\right)\in\mathbb{R}\left[x_{1},\dots,x_{n}\right]$. Let
\[\sigma\left(x_{1},\dots,x_{n}\right)=\begin{cases}  P_{1}\left(x_{1},\dots,x_{n}\right),\\ \hspace{1.2cm}\vdots \\ P_{m}\left(x_{1},\dots,x_{n}\right). \end{cases}\]
The system of equations $\sigma\left(x_{1},\dots,x_{n}\right)=0$ is partition regular near zero on $S$ if and only there exists a $\iota_{\sigma}$-ultrafilter in $0^{+}(S)$.\end{proposition}

We will use two known properties of $\iota_{\sigma}$-ultrafilters. The first, whose routine proof can be found for example in \cite[Example 5.6]{monads}, involves homogeneous systems, namely those systems $\sigma\left(x_{1},\dots,x_{n}\right)$ of polynomial equations with real coefficients such that for all $a,b_{1},\dots,b_{n}\in \R$ one has that $\sigma\left(b_{1},\dots,b_{n}\right)=0$ if and only if $\sigma\left(ab_{1},\dots,ab_{n}\right)=0$.

\begin{theorem}\label{bil id} Let $(S,\cdot)$ be a subsemigroup of $\left(\R^{+},\cdot\right)$. Let $P_{1}\left(x_{1},\dots,x_{n}\right),\dots,P_{m}\left(x_{1},\dots,x_{n}\right)\in \R\left[x_{1},\dots,x_{n}\right]$ be homogeneous. Let
\[\sigma\left(x_{1},\dots,x_{n}\right)=\begin{cases}  P_{1}\left(x_{1},\dots,x_{n}\right),\\ \hspace{1.2cm}\vdots \\ P_{m}\left(x_{1},\dots,x_{n}\right). \end{cases}\]
Assume that the system of equations $\sigma\left(x_{1},\dots,x_{n}\right)=0$ is partition regular on $S$. Then the set
\[I_{\sigma}=\{\U\in\beta S\mid \U \ \text{is a} \ \iota_{\sigma}\text{-ultrafilter}\}\] is a closed bilateral ideal in $\left(\beta S,\odot\right)$.\end{theorem}

The second result, which is just a reformulation of \cite[Lemma 2.1]{advances}, allows us to mix different partition regular systems to produce new ones. We give an explicit proof of the present formulation.

\begin{lemma}\label{lemmasystem}Let $(S,\cdot)$ be a subsemigroup of $\left(\R^{+},\cdot\right)$. Let $P_{1}\left(x_{1},\dots,x_{n}\right)$,$\dots$,$P_{m}\left(x_{1},\dots,x_{n}\right)\in \R\left[x_{1},\dots,x_{n}\right], Q_{1}\left(y_{1},\dots,y_{l}\right),\dots,$ $Q_{t}\left(y_{1},\dots,y_{l}\right)\in \R\left[y_{1},\dots,y_{l}\right]$. Let $\U\in\beta S$ be a witness of the partion regularity of the systems of equations $\sigma_{1}\left(x_{1},\dots,x_{n}\right)=0$, $\sigma_{2}\left(y_{1},\dots,y_{l}\right)=0$, where 
\[\sigma_{1}\left(x_{1},\dots,x_{n}\right)=\begin{cases}  P_{1}\left(x_{1},\dots,x_{n}\right),\\ \hspace{1.2cm}\vdots \\ P_{m}\left(x_{1},\dots,x_{n}\right) \end{cases}\]
and
\[\sigma_{2}\left(y_{1},y_{2}\dots,y_{l}\right)=\begin{cases}  Q_{1}\left(y_{1},\dots,y_{l}\right),\\ \hspace{1.2cm}\vdots \\ Q_{t}\left(y_{1},\dots,y_{l}\right). \end{cases}\]
Then $\U$ witnesses also the partition regularity of $\sigma_{3}\left(x_{1},\dots,x_{n},y_{1},\dots,y_{l}\right)=0$, where
\[\sigma_{3}\left(x_{1},\dots,x_{n},y_{1},\dots,y_{l}\right)=\begin{cases}  P_{1}\left(x_{1},\dots,x_{n}\right),\\ \hspace{1.2cm}\vdots \\ P_{m}\left(x_{1},\dots,x_{n}\right), \\Q_{1}\left(y_{1},\dots,y_{l}\right),\\ \hspace{1.2cm}\vdots \\ Q_{t}\left(y_{1},\dots,y_{l}\right),\\ x_{1}-y_{1}.\end{cases}\]
\end{lemma}

\begin{proof} As $\U$ is a $\iota_{\sigma_{1}}$-ultrafilter, necessarily for all $A\in\U$ the set 
\[I_{A}:=\{a\in A\mid \exists a_{2},\dots,a_{n}\in A \ \text{such that} \ \sigma_{1}\left(a, a_{2},\dots,a_{n}\right)=0\}\in\U,\]
as otherwise $B=A\setminus I_{A}$ would belong to $\U$, but $B$ does not contain any solution to $\sigma_{1}\left(x_{1},\dots,x_{n}\right)=0$ by construction. In a similar way, as $\U$ is a $\iota_{\sigma_{2}}$-ultrafilter, necessarily for all $A\in\U$ the set 
\[J_{A}:=\{a\in A\mid \exists a_{2},\dots,a_{l}\in A \ \text{such that} \ \sigma_{2}\left(a, a_{2},\dots,a_{l}\right)=0\}\in\U.\]
Hence $I_{A}\cap J_{A}$ is nonempty, as $I_{A}\cap J_{A}\in\U$, and $I_{A}\cap J_{A}$ contains solutions to $\sigma_{3}\left(x_{1},\dots,x_{n},y_{1},\dots,y_{l}\right)=0$ by construction. \end{proof}

Finally, we close this section with a simple known observation about ultrafilters in $\beta\mathbb{R}$, namely\footnote{In \cite[Definition 2.4]{HL}, the ultrafilters living at a finite point were also called bounded and defined as those ultrafilters containing a bounded set, but it is a routine proof to show that this definition coincides with ours; ultrafilters that live at infinity were called unbounded. In \cite{monads}, ultrafilters that live at infinity were called infinite, whilst the non-principal ultrafilters living at a finite point were called quasi-principal. This naming comes from the identification of ultrafilters with nonstandard points in enlarged extensions $^{\ast}\R$ of $\R$, done by identifying every ultrafilter $\U\in\beta\R$ with $\mu(\U)=\bigcap_{A\in\U}\,^{\ast}A$. Via this identification, one has that $\U$ lives at $r\in\R$ if and only if $\mu(\U)$ consists solely of finite hyperreals at an infinitesimal distance from $r$, and $\U$ lives at infinity if and only if $\mu(\U)$ consists solely of infinite hyperreals.} that ultrafilters in $\beta\mathbb{R}$ are of two kinds:

\begin{itemize}
\item those that live at a finite point, namely those ultrafilters $\U\in\bR$ for which there exists $r\in\mathbb{R}$ such that for all $\varepsilon>0$ the set $(r-\varepsilon, r+\varepsilon)\in\U$. In this case, we say that $\U$ is infinitesimaly close to $r$;
\item those that live at infinity, namely all those ultrafilters that do not live at a finite point; equivalently, $\U$ lives at infinite if and only if for every $r\in\mathbb{R}$ the set $\{x\in\R\mid |x|\geq |r|\}\in\U$. 
\end{itemize}

Of course, principal ultrafilters live at a finite point, but the converse is not true. (It is true in $\Z$).

\section{Partition regularity of polynomial systems near zero}\label{applicazioni}

We now want to prove results about the partition regularity near zero of polynomial systems in $(S,\cdot)$ for $S$ piecewise syndetic in $((0,1),\cdot)$ or $S$ $\mathbb{Q}$-infinitesimal. Most of these results follow from quite simple observations about ultrafilters.

\begin{theorem}\label{right ideal} If $S$ is a $\mathbb{Q}$-infinitesimal semigroup, then $0^{+}(S)$ is a $\beta\mathbb{Q}^{+}$ left ideal, in the sense that for all $\U\in 0^{+}(S),\V\in\beta\mathbb{Q}^{+}$, $\V\odot\U\in 0^{+}(S)$. \end{theorem}

\begin{proof} Let $\U\in 0^{+}(S)$, let $\V\in\bR^{+}$ and let $\varepsilon>0$. We have to prove that $(0,\varepsilon)\cap S\in\V\odot\U$. Suppose instead that the complement $A=\R^{+}\setminus ((0,\varepsilon)\cap S)\in\V\odot\U$. Then $\{q\in\R^{+}\mid \{s\in S\mid qs\in A\}\in\U\}\in\V$. As $\Q^{+}\in\V$, pick $q\in\Q^{+}$ such that $\{s\in S\mid qs\in A\}\in\U$. Also $\left(0,\frac{\varepsilon}{q}\right)\cap S\in\U$ as $\U\in 0^{+}(S)$, so pick $s\in\left(0,\frac{\varepsilon}{q}\right)\cap S$ such that $qs\in A$. Then $qs\in (0,\varepsilon)\cap S$, a contradiction. \end{proof}

\begin{corollary}\label{eccolo2} Let $P_{1}\left(x_{1},\dots,x_{n}\right)$, $\dots$, $P_{m}\left(x_{1},\dots,x_{n}\right)\in\mathbb{R}\left[x_{1},\dots,x_{n}\right]$. Assume that the system 
\begin{equation}\label{Sigma} \sigma\left(x_{1},\dots,x_{n}\right):=\begin{cases}  P_{1}\left(x_{1},\dots,x_{n}\right),\\ \hspace{1.2cm}\vdots \\ P_{m}\left(x_{1},\dots,x_{n}\right), \end{cases}\end{equation}
is homogeneous. Then the following facts hold:
\begin{enumerate} 
\item if $\sigma\left(x_{1},\dots,x_{n}\right)=0$ is partition regular on $\R^{+}$, then \[\overline{K\left(\bzero,\odot\right)}\subseteq\{\U\in\bzero\mid \U \ \hbox{is a }\iota_{\sigma}\hbox{-ultrafilter}\};\]
\item if $(S,\cdot)$ is a $\mathbb{Q}$-infinitesimal semigroup and $\sigma\left(x_{1},\dots,x_{n}\right)=0$ is partition regular on $\Q^{+}$, then \[\overline{K\left(\beta S,\odot\right)}\subseteq\{\U\in\beta S\mid \U \ \hbox{is a }\iota_{\sigma}\hbox{-ultrafilter}\}.\]
\end{enumerate}\end{corollary}

\begin{proof} By Theorem \ref{bil id}, the set $I_{\sigma}=\{\U\in\beta\R^{+}\mid \U$ is a $\iota_{\sigma}$-ultrafilter\} is a closed bilateral ideal in $\left(\beta\R^{+},\odot\right)$, hence $\overline{K\left(\beta\R^{+},\odot\right)}\subseteq I_{\sigma}$. It sufficies to show that $K\left(\beta(0,1),\odot\right)=\beta(0,1)\cap K\left(\beta\R^{+},\odot\right)$, as if this holds then
$$\overline{K\left(\beta(0,1),\odot\right)}\subseteq \beta(0,1)\cap I_{\sigma}=\{\U\in\beta (0,1)\mid \U \ \hbox{is a }\iota_{\sigma}\hbox{-ultrafilter}\}.$$
By \cite[Theorem 1.65(1)]{rif12}, to prove that $K\left(\beta(0,1),\odot\right)=\beta(0,1)\cap K\left(\beta\R^{+},\odot\right)$ it sufficies to show that $\beta(0,1)\cap K\left(\beta\R^{+},\odot\right)\neq\emptyset$. Let $\U\in 0^{+}$. We claim that $\beta\R^{+}\odot \U\subseteq\beta(0,1)$. In fact, let $\V\in\beta\R^{+}$. If $r\in\R^{+}$, $\left(0,\frac{1}{r}\right)\in\U$, and $\left(0,\frac{1}{r}\right)\subseteq\{s\in\R^{+}\mid rs\in (0,1)\}$, therefore $(0,1)\in\V\odot\U$. Hence $\beta(0,1)$ contains a left ideal of $\left(\beta\R^{+},\odot\right)$, so $\beta(0,1)\cap K\left(\beta\R^{+},\odot\right)\neq \emptyset$ as required.

(2) By exchanging $\R^{+}$ with $\Q^{+}$, $(0,1)$ with $S$ and, consequently, $0^{+}$ with $0^{+}(S)$, the same proof as above works. We prove explicitly the only point where $S$ being $\Q$-infinitesimal is used, namely that if $\U\in 0^{+}(S)$ then $\beta\Q^{+}\odot\U\subseteq \beta S$. Let $\V\in\beta\Q^{+}$. If $q\in\Q^{+}$, $\left(0,\frac{1}{q}\right)\cap S\in\U$ as $\U\in 0^{+}(S)$, and $\left(0,\frac{1}{q}\right)\cap S=\{s\in S\mid qs\in S\}$ as $S$ is $\Q$-infinitesimal. Therefore $\V\odot\U\in\beta S$ and we can proceed as in (1). \end{proof}

\begin{corollary}\label{eccolo3} Let $P_{1}\left(x_{1},\dots,x_{n}\right)$, $\dots$, $P_{m}\left(x_{1},\dots,x_{n}\right)\in\mathbb{R}\left[x_{1},\dots,x_{n}\right]$. Assume that the system 
\begin{equation}\label{Sigma} \sigma\left(x_{1},\dots,x_{n}\right):=\begin{cases}  P_{1}\left(x_{1},\dots,x_{n}\right),\\ \hspace{1.2cm}\vdots \\ P_{m}\left(x_{1},\dots,x_{n}\right) \end{cases}\end{equation}
is homogeneous. Then the following facts hold:
\begin{enumerate} 
\item if $\sigma\left(x_{1},\dots,x_{n}\right)=0$ is partition regular on $\R^{+}$, then every set $A$ that is piecewise syndetic in $((0,1),\cdot)$ contains a solution to $\sigma\left(x_{1},\dots,x_{n}\right)=0$;
\item if $(S,\cdot)$ is a $\mathbb{Q}$-infinitesimal semigroup and $\sigma\left(x_{1},\dots,x_{n}\right)=0$ is partition regular on $\Q^{+}$, then every set $A$ that is piecewise syndetic in $(S,\cdot)$ contains a solution to $\sigma\left(x_{1},\dots,x_{n}\right)=0$.
\end{enumerate}\end{corollary}

\begin{proof} These facts follows directly from Corollary \ref{eccolo2} and the fact that for any semigroup $S$ and any $A\subseteq S$ one has that $\overline{A}\cap K(\beta S,\odot)\neq \emptyset$ if and only if $A$ is piecewise syndetic in $(S,\cdot)$ (see \cite[Theorem 4.40]{rif12}).\end{proof}

\begin{example} Theorem $9$ in \cite{indiani}, namely the fact that piecewise syndetic sets in $(S\cap (0,1),\cdot)$ for $S$ a $HL$-semigroup, are both $AP$-rich near zero and $IP_{0}$-sets near zero, is an immediate consequence of Corollary \ref{eccolo3}, as all systems that describe the properties of being $AP$-rich and being $IP_{0}$ are homogeneous and partition regular on $\mathbb{Q}^{+}$ by Theorem \ref{RadoS}. Notice that Corollary \ref{eccolo3} tells us more: if $A$ is any matrix with coefficients in $\mathbb{Q}$ with the columns property, then the system $A\vec{x}$ is automatically partition regular near zero and solvable in any piecewise syndetic subset of a $\mathbb{Q}$-infinitesimal semigroup $(S,\cdot)$.\end{example}

\begin{example}\label{razionali} As the partition regularity notions on $\R^{+}$ and $\Q^{+}$ are not the same, one has to be careful when applying Corollary \ref{eccolo3}. For example, let us consider Fermat's polynomial $P(x,y,z):=x^{3}+y^{3}-z^{3}$. By Theorem \ref{Rado2}, $P(x,y,z)=0$ is partition regular on $\R^{+}$ hence, as it is homogeneous, it is partition regular near zero and solvable in any piecewise syndetic subset of $((0,1),\cdot)$. However, it is not solvable in all piecewise syndetic subsets of $\mathbb{Q}$-infinitesimal semigroups: for example, it is not solvable in $\mathbb{Q}\cap (0,1)$. This gives another proof of the fact that $\mathbb{Q}\cap (0,1)$ is not piecewise syndetic in $((0,1),\cdot)$.\end{example}

The polynomial $a+b-cd$ is not homogeneous, so its partition regularity near zero can not be directly deduced from Theorem \ref{right ideal} and its corollaries. However, we can prove that it is partition regular near zero (and construct ultrafilters that witness its partition regularity) following the 
methods first introduced in \cite{mio}, using nonstandard analysis, and then developed also in \cite{advances} by purely standard methods. We will use the following known simple fact (which is a trivial consequence of idempotency, see for example \cite[Theorem 5.12]{rif12}):

\begin{proposition}\label{IPO} Let $(S,\cdot)$ be a subsemigroup of $\left(\R^{+},\cdot\right)$ and let $\U$ be an idempotent in $(\beta S,\odot)$. Let $n\in \N$ and let $P\left(x,y_{1},\dots,y_{n}\right):= x-\prod_{j=1}^{n} y_{j}$. Then $\U$ is a $\iota_{P}$-ultrafilter. \end{proposition}

\begin{proposition} The polynomial $P(a,b,c,d):=a+b-cd$ is partition regular near zero; more precisely, any idempotent ultrafilter $\U$ in either $\overline{K(\bzero,\odot)}$ or $\overline{K(\beta S,\odot)}$ for $S$ $\mathbb{Q}$-infinitesimal semigroup is a $\iota_{P}$-ultrafilter. \end{proposition}

\begin{proof} Let $\U=\U\odot\U$ be an idempotent ultrafilter in either $\overline{K\left(\bzero,\odot\right)}$ or $\overline{K(\beta S,\odot)}$. From Corollary \ref{eccolo2} we know that $\U$ witnesses the partition regularity of $a+b=x$, as this equation is partition regular on $\Q^{+}$. As $\U$ is idempotent, by Proposition \ref{IPO} $\U$ witnesses the partition regularity of the equation $y=cd$. So by Lemma \ref{lemmasystem} we have that $\U$ witnesses the partition regularity of the system  
\[\begin{cases}  a+b=x; \\ y=cd; \\ x=y,\end{cases}\]
hence $\U$ witnesses the partition regularity of $a+b-cd=0$.\end{proof}

Let us notice that the above proof actually shows more, as it tells us that the color of $a+b, cd$ is the same as that of $a,b,c,d$. This proof can be generalized to show the analogue of Theorem \ref{lev} near zero.

\begin{theorem}\label{lev2} Let $n\geq 2$ be a natural number, let $R\left(x_{1},\dots,x_{n}\right)=\sum\limits_{i=1}^{n} c_{i}x_{i}\in\R\left[x_{1},\dots,x_{n}\right]$ be partition regular on $\R^{+}$, let $m$ be a positive natural number, and let $F_{1},\dots,F_{n}\subseteq\{1,\dots,m\}$. Let\footnote{The sets $Q_{F_{i}}$ are defined in Definition \ref{Qi}.}
\begin{equation}\label{pollo} P\left(x_{1},\dots,x_{n},y_{1},\dots,y_{m}\right)=\sum\limits_{i=1}^{n} c_{i}x_{i}Q_{F_{i}}\left(y_{1},\dots,y_{m}\right). \end{equation}
Then $P\left(x_{1},\dots,x_{n},y_{1},\dots,y_{m}\right)$ is partition regular near zero. More precisely:
\begin{enumerate}
\item every idempotent ultrafilter in $\overline{K(\bzero,\odot)}$ is a $\iota_{P}$-ultrafilter;
\item if $R\left(x_{1},\dots,x_{n}\right)$ is partition regular on $\Q^{+}$ and $S$ is a $\mathbb{Q}$-infinitesimal semigroup then every idempotent ultrafilter in $\overline{K(\beta S,\odot)}$ is a $\iota_{P}$-ultrafilter.
\end{enumerate}\end{theorem}

\begin{proof} This result can be proven following the same ideas of the proof of Theorem 3.3 in \cite{mio}, if one wants to use nonstandard methods, or of Theorem 2.10 in \cite{advances}, if one wants to use purely standard arguments based on ultrafilters. We adapt here the proof of Theorem 2.10 in \cite{advances} (which talked about the partition regularity on $\N$ and was more complicated as we handled also the injectivity properties of the sets of solutions of equation \ref{pollo}) to our present case. 

Let $\U$ be an idempotent ultrafilter in $\overline{K(\bzero,\odot)}$ or $\overline{K(\beta S,\odot)}$. In both cases, our hypothesis on $R\left(x_{1},\dots,x_{n}\right)$ ensures that $\U$ is a $\iota_{R}$-ultrafilter. Given $A\in\U$, set $B_0=A$ and
inductively define 
$$B_{k}=\{x\in B_{k-1}\mid \{y\in B_{k-1}\mid 
x\cdot y\in B_{k-1}\}\in \U\}.$$
Trivially we have that $B_{m}\subseteq B_{m-1}\subseteq\dots
\subseteq B_{1}\subseteq B_0=A$ and, as $\U$ is multiplicatively idempotent and $B_0\in\U$, 
it is immediate that all sets $B_{1},\dots, B_{m}\in\U$.
Since $\U$ is a $\iota_{R}$-ultrafilter, there exist $a_{1},\dots,a_{n}\in B_{m}$ such that
$R\left(a_{1},\dots,a_{n}\right)=0$. We now claim that there exist $b_1,\ldots,b_m\in A$
such that: 
\begin{enumerate}
\item 
$b_k\in B_{m-k}$ for every $k=1,\ldots,m$,
\item 
$a_{i}\cdot\prod_{j\in G}b_{j}\in B_{m-\max G}$
for every $i$ and for every set $G\subseteq\{1,\dots,m\}$.\footnote
{~We agree that $\prod_{j\in G} b_{j} =1$ and $\max G=0$ when $G=\emptyset$.}
\end{enumerate}

We define $b_{k}$ inductively for $k\leq m$.

Let $k=1$.
As $a_i\in B_m$ for every $i=1,\dots,n$, we have that
$$C_{i}=\{y\in B_{m-1}\mid a_{i}\cdot y\in B_{m-1}\}\in\U.$$
Pick $b_1\in C_{1}\cap\dots\cap C_{n}\in\U$. 
Trivially, $b_1\in B_{m-1}$ and, for every $i\leq n$, we have
$a_i\cdot\prod_{j\in\{1\}}b_j=a_{i}\cdot b_{1}\in B_{m-1}=B_{m-\max\{1\}}$,
and $a_i\cdot\prod_{j\in\emptyset}b_j=a_i\in B_m=B_{m-\max\emptyset}$.

Now let us prove the inductive step. Assume that numbers $b_{1},\dots,b_{k}$ where $k\le m-1$ fulfilling the properties of the claim have been defined. To define $b_{k+1}$, we observe that for every set $G\subseteq\{1,\dots,k\}$ and
for every $i$, by the inductive hypothesis 
$a_{i}\cdot\prod_{j\in G}b_j\in B_{m-\max G}$, and hence
$$C_{G,i}=\left\{y\in B_{m-\max G-1}\,\Big|\, 
a_{i}\cdot \prod_{j\in G}b_j\cdot y \in B_{m-\max G-1}\right\}\in\U.$$
Let $b_{k+1}\in\,\bigcap_{i=1}^{n}\left(\bigcap_{G\subseteq\{1,\dots,k\}}\!\!
C_{G,i}\right)\in\U.$

Notice that every $C_{G,i}\subseteq B_{m-\max G-1}\subseteq B_{m-(k+1)}$,
and so $b_{k+1}\in B_{m-(k+1)}$. To prove that $b_{k+1}$ has the desired multiplicative properties, let $G\subseteq\{1,\dots,k+1\}$.
If $G\subseteq\{1,\ldots,k\}$ then, by the inductive hypothesis,
$a_i\cdot\prod_{j\in G}b_j\in B_{m-\max G}$ for every $i$.
If $k+1\in G$, let $G'=G\setminus\{k+1\}$.
For every $i$,
by the inductive hypothesis on $G'$, we know that
$$a_i\cdot\prod_{j\in G'}b_j\in B_{m-\max G'}\subseteq B_{m-k},$$
so, as $b_{k+1}\in C_{G',i}$, we deduce that
$$a_i\cdot\prod_{j\in G}b_j=a_i\cdot\prod_{j\in G'}b_j\cdot 
b_{k+1}\in B_{m-\max G'-1}\subseteq B_{m-k-1}\subseteq B_{m-\max G}.$$
This proves the claim.

Now, for $i=1,\dots,n$ let
$$d_{i}:=a_{i}\cdot \prod_{j\in F_{i}^{c}} b_{j}.$$

Then $d_{1},\dots,d_{n},b_{1},\dots,b_{m}$ are elements of $A$ with $P\left(d_{1},\dots,d_{n},b_{1},\dots,b_{m}\right)=0$.
Indeed, by the claim, we have that $d_i\in B_{m-\max F_i^c}\subseteq A$ and
$b_j\in B_{m-j}\subseteq A$. Moreover,
$$\sum_{i=1}^{n} c_{i}\,d_i\!\left(\prod_{j\in F_{i}}b_{j}\right)=
\sum_{i=1}^{n} c_{i}\,a_i\left(\prod_{j\in F_{i}^c}b_{j}\right)\!\left(\prod_{j\in F_{i}}b_{j}\right)=
\left(\prod_{j=1}^{m}b_{j}\right)\!\left(\sum_{i=1}^{n} c_{i}a_{i}\right)=0.$$
\end{proof}

By putting together Theorem \ref{lev2}, Corollary \ref{eccolo2} and Lemma \ref{lemmasystem}, we obtain our final result about polynomial systems that are partition regular near zero.

\begin{theorem}\label{main} Let $\U$ be an idempotent ultrafilter in $\overline{K(\bzero,\odot)}$ (resp., let $\U$ be an idempotent ultrafilter in $\overline{K(\beta S,\odot)}$ for $S$ a $\mathbb{Q}$-infinitesimal semigroup). Let $\mathcal{C}_{\U}$ be the set of polynomial systems whose partition regularity is witnessed by $\U$. Then $\mathcal{C}_{\U}$ includes:
\begin{enumerate} 
\item all partition regular homogeneous systems on $\R^{+}$ (resp. all partition regular homogeneous systems on $\Q^{+}$);
\item all equations of the form 
\[P\left(x_{1},\dots,x_{n},y_{1},\dots,y_{m}\right)=\sum\limits_{i=1}^{n} a_{i}x_{i}Q_{F_{i}}\left(y_{1},\dots,y_{m}\right)\] where $\sum\limits_{i=1}^{n} a_{i}x_{i}\in\R\left[x_{1},\dots,x_{n}\right]$ is partition regular on $\R^{+}$ (resp. on $\Q^{+}$) and $F_{1},\dots,F_{n}\subseteq\{1,\dots,m\}$.
\end{enumerate}
Moreover, if
\[\sigma_{1}\left(x_{1},\dots,x_{n}\right)=\begin{cases}  P_{1}\left(x_{1},\dots,x_{n}\right),\\ \hspace{1.2cm}\vdots \\ P_{m}\left(x_{1},\dots,x_{n}\right) \end{cases}\]
and
\[\sigma_{2}\left(y_{1},\dots,y_{l}\right)=\begin{cases}  Q_{1}\left(y_{1},\dots,y_{l}\right),\\ \hspace{1.2cm}\vdots \\ Q_{t}\left(y_{1},\dots,y_{l}\right) \end{cases}\]
belong to $\mathcal{C}_{\U}$, then also
\[\sigma_{3}\left(x_{1},\dots,x_{n},y_{1},\dots,y_{l}\right)=\begin{cases}  P_{1}\left(x_{1},\dots,x_{n}\right),\\ \hspace{1.2cm}\vdots \\ P_{m}\left(x_{1},\dots,x_{n}\right), \\Q_{1}\left(y_{1},\dots,y_{l}\right),\\ \hspace{1.2cm}\vdots \\ Q_{t}\left(y_{1},\dots,y_{l}\right),\\ x_{1}-y_{1}\end{cases}\]
belongs to $\mathcal{C}_{\U}$.

\end{theorem}

\begin{example} The equations $x_{1}-x_{2}y_{1}-2x_{2}y_{1}y_{2}=0$ and $x^{2}+y^{2}-z^{2}=0$ are partition regular near zero, hence also the equation $\left(x_{2}y_{1}+2x_{2}y_{1}y_{2}\right)^{2}+y^{2}-z^{2}$ is partition regular near zero.\end{example}

To conclude, we show that the notions of partition regularity on $\mathbb{R}^{+}$ and partition regularity near zero are closely related.

\begin{theorem}\label{inverti} Let $P_{1}\left(x_{1},\dots,x_{n}\right)$, $\dots$, $P_{m}\left(x_{1},\dots,x_{n}\right)\in\mathbb{R}\left[x_{1},\dots,x_{n}\right]$. Assume that the system 
\begin{equation}\label{Sigma} \sigma\left(x_{1},\dots,x_{n}\right):=\begin{cases}  P_{1}\left(x_{1},\dots,x_{n}\right),\\ \hspace{1.2cm}\vdots \\ P_{m}\left(x_{1},\dots,x_{n}\right), \end{cases}\end{equation}

does not have any constant solution in $\R^{+}$, namely for all $r\in\mathbb{R}\setminus\{0\}$ $\sigma(r,\dots,r)\neq 0$. Then the following facts are equivalent:
\begin{enumerate}
\item the system $\sigma\left(x_{1},\dots,x_{n}\right)=0$ is partition regular near zero;
\item the system $\sigma\left(\frac{1}{x_{1}},\dots,\frac{1}{x_{n}}\right)=0$ is partition regular on $\R^{+}$.
\end{enumerate}
\end{theorem}

\begin{proof} The implication $(1)\Rightarrow (2)$ is a particular case of Theorem 2.1 in \cite{lefmann}, as the function $f(x)=\frac{1}{x}$ is a bijection on $\mathbb{R}^{+}$ and the partition regularity near zero is a particular case of the partition regularity on $\R$.

Let us prove that $(2)\Rightarrow (1)$. As $\sigma\left(\frac{1}{x_{1}},\dots,\frac{1}{x_{n}}\right)=0$ is partition regular on $\R^{+}$, it has a witness of its partition regularity in $\beta\R^{+}$. 

If this witness $\U$ lives at a finite point, it is by definition infinitesimaly close to some $r\in\mathbb{R}^{+}$. Then it must be $\sigma(r,\dots,r)=0$. In fact, if not, by continuity there exists $\varepsilon>0$ such that for all $r_{1},\dots,r_{n}\in(r-\varepsilon,r+\varepsilon)$ one has $\sigma\left(r_{1},\dots,r_{n}\right)\neq 0$. But then the system $\sigma\left(x_{1},\dots,x_{n}\right)$ does not have any solution in $\left(r-\varepsilon,r+\varepsilon\right)\in\U$, so $\U$ cannot witness its partition regularity. As we assumed that $\sigma\left(x_{1},\dots,x_{n}\right)=0$ does not have any constant solution in $\R^{+}$, this proves that $\U$ must live at infinity. Let
\[\U^{-1}:=\{A\subseteq\R^{+}\mid A^{-1}\in\U\},\]

where $A^{-1}:=\{a\in\R^{+}\mid \frac{1}{a}\in A\}$.

\

{\bfseries Claim:} The following facts hold:
\begin{itemize}
\item $\U^{-1}\in 0^{+}$;
\item $\U^{-1}$ witnesses the partition regularity of $\sigma\left(x_{1},\dots,x_{n}\right)=0$.
\end{itemize}

Clearly, our thesis follows immediately from the above claim by Proposition \ref{ultra part zero}. Let us prove the claim.

To prove the first fact, let $\varepsilon>0$. Then $\left(\frac{1}{\varepsilon},+\infty\right)\in\U$ as $\U$ lives at infinity, so $(0,\varepsilon)\in\U^{-1}$ by definition.

To prove the second fact, take any $A\in\U^{-1}$. As $\U$ witnesses the partition regularity of $\sigma\left(\frac{1}{x_{1}},\dots,\frac{1}{x_{n}}\right)=0$, there are $b_{1},\dots,b_{n}\in A^{-1}$ such that $\sigma\left(\frac{1}{b_{1}},\dots,\frac{1}{b_{n}}\right)=0$. But for all $i\leq n$ $\frac{1}{b_{i}}\in A$, which proves that $A$ contains a solution to $\sigma\left(x_{1},\dots,x_{n}\right)=0$. \end{proof}

Notice that, in Theorem \ref{inverti}, the hypothesis on the existence of nonconstant solutions could be substituted with the request that the partition regularity of the system $\sigma\left(\frac{1}{x_{1}},\dots,\frac{1}{x_{n}}\right)=0$ is witnessed by an ultrafilter that lives at infinity.

\end{document}